\documentclass[a4paper,12pt]{article}
\usepackage{amssymb,amsmath,amsthm,latexsym}
\usepackage{amsfonts}
	\usepackage{amsfonts}
\usepackage{graphicx}
\usepackage[pdftex,bookmarks,colorlinks=false]{hyperref}
\usepackage{verbatim}
\usepackage{caption}
\usepackage{subcaption}

\newtheorem{theorem}{Theorem}[section]

\newtheorem{definition}[theorem]{Definition}

\newtheorem{lemma} [theorem]{Lemma}

\newtheorem{problem}[theorem]{Problem}
\newtheorem{proposition}[theorem]{Proposition}
\newtheorem{remark}[theorem]{Remark}

\renewcommand{\theequation}{\thesubsection.\arabic{equation}}

\title{\bf On Integer Additive Set-Indexers of Graphs}
\author{{\bf N K Sudev \footnote{Department of Mathematics, Vidya Academy of Science \& Technology, Thalakkottukara, Thrissur - 680501, email: {\em sudevnk@gmail.com}}} and {\bf K A Germina\footnote{Department of Mathematics, School of Mathematical \& Physical Sciences, Central University of Kerala, Kasaragod, email:{\em srgerminaka@gmail.com}}}}
\date{}
\begin{document}
\maketitle

\begin{abstract}
A set-indexer of a graph $G$ is an injective set-valued function $f:V(G) \rightarrow2^{X}$ such that the function $f^{\oplus}:E(G)\rightarrow2^{X}-\{\emptyset\}$ defined by $f^{\oplus}(uv) = f(u ){\oplus} f(v)$ for every $uv{\in} E(G)$ is also injective, where $2^{X}$ is the set of all subsets of $X$ and $\oplus$ is the symmetric difference of sets. An integer additive set-indexer is defined as an injective function $f:V(G)\rightarrow 2^{\mathbb{N}_0}$ such that the induced function $g_f:E(G) \rightarrow 2^{\mathbb{N}_0}$ defined by $g_f (uv) = f(u)+ f(v)$ is also injective. A graph $G$ which admits an IASI is called an IASI graph. An IASI $f$ is said to be a {\em weak IASI} if $|g_f(uv)|=max(|f(u)|,|f(v)|)$ and an IASI $f$ is said to be a {\em strong IASI} if $|g_f(uv)|=|f(u)| |f(v)|$ for all $u,v\in V(G)$.  In this paper, we study about certain characteristics of inter additive set-indexers.
\end{abstract}

\section{Introduction}

For all  terms and definitions, not defined specifically in this paper, we refer to \cite{FH} and for more about graph labeling, we refer to \cite{JAG}. Unless mentioned otherwise, all graphs considered here are simple, finite and have no isolated vertices.

For a $(p,q)$- graph $G=(V,E)$ and a non-empty set $X$ of cardinality $n$, a {\em set-indexer} of $G$ is defined in \cite{A1} as an injective set-valued function $f:V(G) \rightarrow2^{X}$ such that the function $f^{\oplus}:E(G)\rightarrow2^{X}-\{\emptyset\}$ defined by $f^{\oplus}(uv) = f(u ){\oplus} f(v)$ for every $uv{\in} E(G)$ is also injective, where $2^{X}$ is the set of all subsets of $X$ and $\oplus$ is the symmetric difference of sets. 

\begin{theorem}\label{T-ASI1}
\cite{A1} Every graph has a set-indexer.
\end{theorem}

Let $\mathbb{N}_0$ denote the set of all non-negative integers. For all $A, B \subseteq \mathbb{N}_0$, the sum of these sets is denoted by  $A+B$ and is defined by $A + B = \{a+b: a \in A, b \in B\}$. The set $A+B$ is called the {\em sumset} of the sets $A$ anb $B$. Using the concept of sumsets, an integer additive set-indexer is defined as follows.

\begin{definition}\label{D2}{\rm
\cite{GA} An {\em integer additive set-indexer} (IASI, in short) is defined as an injective function $f:V(G)\rightarrow 2^{\mathbb{N}_0}$ such that the induced function $g_f:E(G) \rightarrow 2^{\mathbb{N}_0}$ defined by $g_f (uv) = f(u)+ f(v)$ is also injective}. A graph $G$ which admits an IASI is called an IASI graph.
\end{definition}

\begin{definition}\label{D3}{\rm
\cite{GS1} The cardinality of the labeling set of an element (vertex or edge) of a graph $G$ is called the {\em set-indexing number} of that element.} 
\end{definition}

\begin{definition}\label{DU}{\rm
\cite{GA} An IASI is said to be {\em $k$-uniform} if $|g_f(e)| = k$ for all $e\in E(G)$. That is, a connected graph $G$ is said to have a $k$-uniform IASI if all of its edges have the same set-indexing number $k$.}
\end{definition}

\begin{definition}\label{D4}{\rm
An element (a vertex or an edge) of graph G, which has the set-indexing number 1, is called a {\em mono-
indexed element} of that graph.}
\end{definition}

In particular, we say that a graph $G$ has an {\em arbitrarily $k$-uniform IASI} if $G$ has a $k$-uniform IASI  for every positive integer $k$.

In \cite{GS2}, the vertex set $V$ of a graph $G$ is defined to be {\em $l$-uniformly set-indexed}, if all the vertices of $G$ have the set-indexing number $l$.

In this paper, we intend to introduce some fundamental notions on IASIs and establish some useful results. 

\section{Integer Additive Set-Indexers of Graphs}

Analogous to Theorem \ref{T-ASI1} we prove the following theorem.

\begin{theorem}
Every graph $G$ admits an IASI.
\end{theorem}
\begin{proof}
Let $G$ be a graph with vertex set $V(G)=\{v_1,v_2,\ldots, v_n\}$. Let $A_1, A_2,\ldots,A_n$ be distinct non-empty subsets of $\mathbb{N}_0$. Let $f:V(G)\to 2^{\mathbb{N}_0}$ defined by $f(v_i)=A_i$. We observe that $f$ is injective. Define $g_f:E(G)\to 2^{\mathbb{N}_0}$ as $g_f(v_iv_j)=\{a_i+b_j:a_i\in A_i, b_j\in A_j\}$. Clearly, $g_f(v_iv_j)$ is a set of non-negative integers. For suitable choices of the sets $A_i$, we observe that $g_f$ is injective. Hence, $f$ is an IASI on $G$.
\end{proof}

Figure \ref{G-IASIG} depicts an IASI graph.

\begin{figure}[h!]
\centering
\includegraphics[scale=0.5]{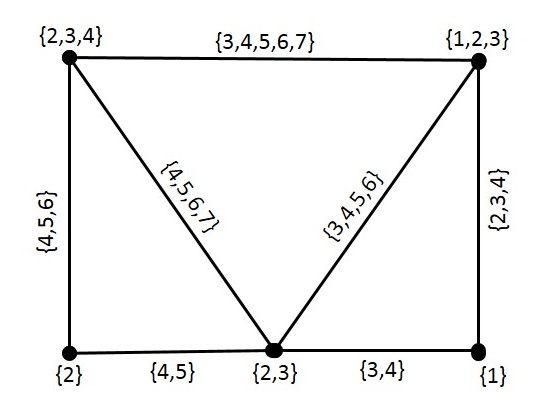}
\caption{}\label{G-IASIG}
\end{figure}

The following results verify the admissibility of IASIs by certain graphs that are associated with a given graph $G$.

\begin{theorem}\label{T-IASI-S}
A subgraph $H$ of an IASI graph $G$ also admits an IASI. 
\end{theorem}
\begin{proof}
Let $G$ be a graph which admits an IASI, say $f$. Let $f|_H$ be the restriction of $f$ to $V(H)$. Then, $g_{f|_H}$ is the corresponding restriction of $g_f$ to $E(H)$. Then, $f|_H$ is an IASI on $H$. Therefore, $H$ is also an IASI graph.
\end{proof}

We call the IASI $f|_H$, defined in Theorem \ref{T-IASI-S} as the {\em induced IASI} of the subgraph $H$ of $G$.

By {\em edge contraction operation} in $G$, we mean an edge, say $e$, is removed and its two incident vertices, $u$ and $v$, are merged into a new vertex $w$, where the edges incident to $w$ each correspond to an edge incident to either $u$ or $v$. We establish the following theorem for the graphs obtained by contracting the edges of a given graph $G$. The following theorem establishes the admissibility of the graphs obtained by contracting the edges of a given IASI graph $G$.

\begin{theorem}
Let $G$ be an IASI graph and let $e$ be an edge of $G$. Then, $G\circ e$ admits an IASI.
\end{theorem}
\begin{proof}
Let $G$ admits a weak IASI. $G\circ e$ is the graph obtained from $G$ by deleting an edge $e$ of $G$ and identifying the end vertices of $e$. Label the new vertex thus obtained, say $w$, by the set-label of the deleted edge. Then, each edge incident upon $w$ has a set-label containing non-negative integers. Hence, $G\circ e$ is an IASI graph. 
\end{proof}

\begin{definition}{\rm
\cite{KDJ} Let $G$ be a connected graph and let $v$ be a vertex of $G$ with $d(v)=2$. Then, $v$ is adjacent to two vertices $u$ and $w$ in $G$. If $u$ and $v$ are non-adjacent vertices in $G$, then delete $v$ from $G$ and add the edge $uw$ to $G-\{v\}$. This operation is known as an {\em elementary topological reduction} on $G$.}
\end{definition}

\begin{theorem}\label{T-IASI-TRG}
Let $G$ be a graph which admits an IASI. Then any graph $G'$, obtained by applying finite number of elementary topological reductions on $G$, admits an IASI. 
\end{theorem}
\begin{proof}
Let $G$ be a graph which admits an IASI, say $f$. Let $v$ be a vertex of $G$ with $d(v)=2$. Then $v$ is adjacent two non adjacent vertices $u$ and $w$ in $G$. Now remove the vertex $v$ from $G$ and introduce the edge $uw$ to $G-{v}$. Let $G'=(G-{v})\cup \{uw\}$. Now $V(G')=V(G-v)\subset V(G)$. Let $f':V(G')\to 2^{\mathbb{N}_0}$ such that $f'(v)=f(v)~ \forall ~v\in V(G')(or V(G-v))$ and the associated function $g_{f'}:E(G')\to 2^{\mathbb{N}_0}$ and defined by 
\[ g_{f'}(e)= \left\{
\begin{array}{l l}	
g_f(e)& \quad \text{if $e\ne uw$}\\
f(u)+f(w)& \quad \text{if $e=uw$}
\end{array} \right.\]
Hence, $f'$ is an IASI of $G'$. 
\end{proof}

Another associated graph of $G$ is the line graph of $G$.  First, we have the definition of the line graph of a given graph $G$ and the admissibility of IASI by the line graph of $G$ is followed.

\begin{definition}{\em 
\cite{DBW} For a given graph $G$, its {\em line graph} $L(G)$ is a graph such that  each vertex of $L(G)$ represents an edge of $G$ and two vertices of $L(G)$ are adjacent if and only if their corresponding edges in $G$ incident on a common vertex in $G$.}
\end{definition}

\begin{theorem}\label{T-IASI-LG}
If $G$ is an IASI graph, then its line graph $L(G)$ is also an IASI graph.
\end{theorem}
\begin{proof}
Since $G$ admits an IASI, say $f$, both $f$ and $g_f$ are injective functions. Let $E(G)=\{e_1,e_2,e_3,\ldots,e_n\}$. For $1\le i \le n$, let $u_i$, be the vertex in $V(L(G))$ corresponding to the edge $e_i$ in $G$. Define $f':V(L(G))\to 2^{\mathbb{N}_0}$ by $f'(u_i)=g_f(e_i), 1\le i \le n$. Clearly, $f'$ is injective. Since each $f'(u_i)$ is a set of non-negative integers, the associated function $g_{f'}:E(L(G))\to 2^{\mathbb{N}_0}$, defined by $g_{f'}(u_iu_j)=f'(u_i)+f'(u_j)$, is also injective and each $g_{f'}(u_iu_j)$ is a set of non-negative integers. Therefore, $f'$ is an IASI of $L(G)$.
\end{proof}

Figure \ref{G-IASIG-LG} depicts the admissibility of an IASI by the line graph $L(G)$ of a weakly IASI graph $G$.

\begin{figure}[h!]
\centering
\includegraphics[scale=0.35]{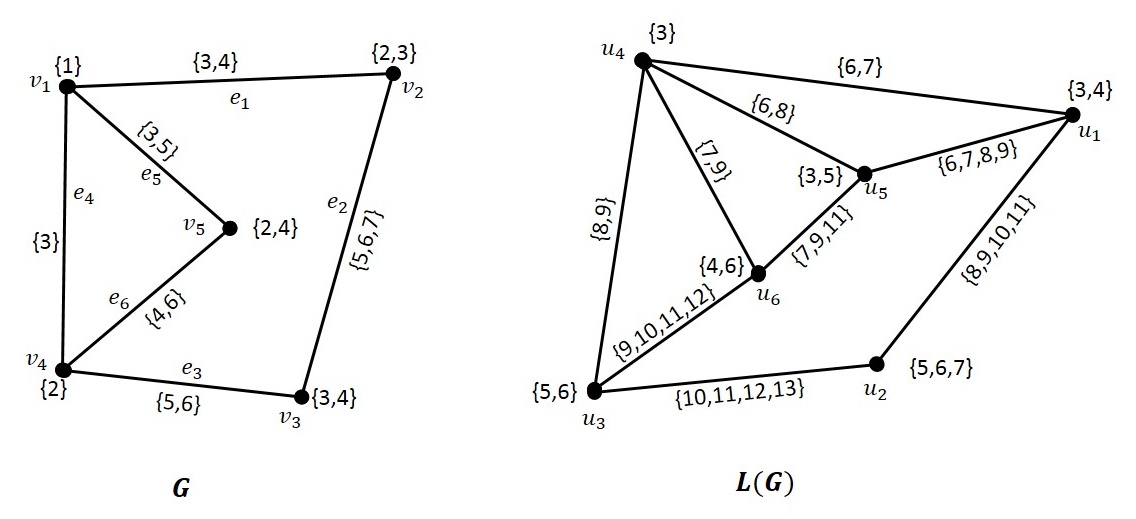}
\caption{}\label{G-IASIG-LG}
\end{figure}

Now we recall the notion of the total graph of a given graph $G$ and hence verify the admissibility of IASI by the total graph of an IASI graph.

\begin{definition}{\rm
\cite{MB} The {\em total graph} of a graph $G$ is the graph, denoted by $T(G)$, is the graph having the property that a one-to one correspondence can be defined between its points and the elements (vertices and edges) of $G$ such that two points of $T(G)$ are adjacent if and only if the corresponding elements of $G$ are adjacent (either  if both elements are edges or if both elements are vertices) or they are incident (if one element is an edge and the other is a vertex). }
\end{definition} 

\begin{theorem}\label{T-IASI-TG}
If $G$ is an IASI graph, then its total graph $T(G)$ is also an IASI graph.
\end{theorem}
\begin{proof}
Since $G$ admits an IASI, say $f$, by the definition of an IASI $f(v),~\forall~ v\in V(G)$ and $g_f(e), ~\forall~ e\in E(G)$ are sets of non-negative integers. Define a map $f':V(T(G))\to 2^{\mathbb{N}_0}$ which assigns the same set-labels of the corresponding elements in $G$ under $f$ to the vertices of $T(G)$. Clearly, $f'$ is injective and each $f'(u_i),~u_i\in V(T(G))$ is a set of non-negative integers. Now, define the associated function $g_f:E(T(G))\to 2^{\mathbb{N}_0}$ defined by $g_{f'}(u_iu_j)= f'(u_i)+f'(u_j),~ u_i,u_j\in V(T(G))$. Then, $g_{f'}$ is injective and each $g_{f'}(u_iu_j)$ is a set of non-negative integers. Therefore, $f'$ is an IASI of $T(G)$. This completes the proof.
\end{proof}

Figure \ref{G-IASIG-TG} illustrates the admissibility of an IASI by the the total graph $T(G)$ of a weakly IASI graph $G$.

\begin{figure}[h!]
\centering
\includegraphics[scale=0.5]{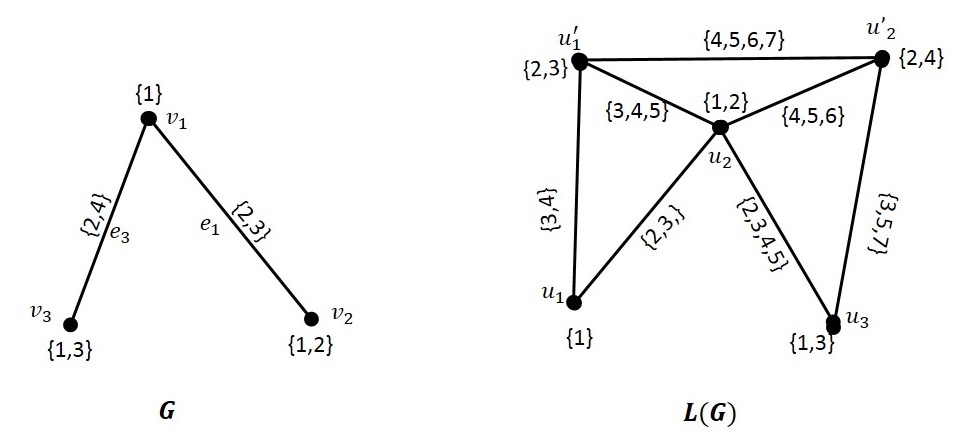}
\caption{}\label{G-IASIG-TG}
\end{figure}

\section{Some Certain Types of IASIs}

\subsection{Cardinality of the Set-Labels of Graphs}

To determine the set-indexing number of an edge of a graph $G$, if the set-indexing numbers of its end vertices are given, is an interesting problem.  All the sets mentioned in this paper are sets of non-negative integers. We denote the cardinality of a set $A$ by $|A|$. Then, we recall the following lemma.

\begin{lemma}\label{L-0}
\cite{GS1} If $A, B\subseteq \mathbb{N}_0$, then $max(|A|,|B|)\le |A+B|\le |A|~|B|$.
\end{lemma}

Let $f$ be an IASI defined on $G$ and let $u$ and $v$ be two adjacent vertices of $G$ labeled by two non-empty sets $A$ and $B$ respectively. Then, the set-label for the edge $uv$ in $G$ is $A+B$. That is, $f(u)=A$ and $f(v)=B$ and $g_f(uv)=A+B$. 

\begin{definition}\label{D-CP}{\rm 
Two ordered pairs $(a,b)$ and $(c,d)$ in $A\times B$ {\em compatible} if $a+b=c+d$. If $(a,b)$ and $(c,d)$ are compatible, then we write $(a,b)\sim (c,d)$}.
\end{definition}

\begin{definition}\label{D-CC}{\rm 
A {\em compatible class} of an ordered pair $(a,b)$ in $A\times B$ with respect to the integer $k=a+b$ is the subset of $A\times B$ defined by $\{(c,d)\in A\times B:(a,b)\sim (c,d)\}$ and is denoted by $[(a,b)]_k$ or $\mathsf{C}_k$.}
\end{definition}

From Definition \ref{D-CC}, we note that $(a,b) \in [(a,b)]_k$, each compatibility classes $[(a,b)]_k$ is non-empty. That is, the minimum number of elements in a compatibility class is $1$. If a compatibility class contains exactly one element, then it is called a {\em trivial class}.

All the compatibility class need not have the same number of elements but can not exceed a certain number. The compatibility classes which contain the maximum possible number of elements are called {\em saturated classes}. The compatibility class that contains maximum elements is called a {\em maximal compatibility class}. Hence, it can be observed that all saturated classes are maximal compatibility classes, but a maximal compatibility class need not be a saturated class. 

The following result provide the maximimum possible number of elements in a saturated class.

\begin{proposition}\label{P-CardCC}
The cardinality of a saturated class in $A\times B$ is $n$, where $n=min(|A|,|B|)$. 
\end{proposition}
\begin{proof}
Let $A=\{a_1,a_2,a_3, \ldots, a_m\}$ and $B=\{b_1,b_2,b_3, \ldots, b_n\}$ be two sets with $m>n$. Arrange the elements of $A+B$ in rows and columns as follows. 

Arrange the elements of $A+\{b_j\}$ in $j$-th row in such a way that equal sums in each row comes in the same column. Hence, we have $n$ rows in this arrangement in which each column corresponds to a compatibility class in $(A,B)$. Therefore, a compatibility class can have a maximum of $n$ elements.
\end{proof}

\begin{definition}\label{D-CI}{\rm 
The number of distinct compatibility classes in $A\times B$ is called the {\em compatibility index} of the pair of sets $(A,B)$ and is denoted by $\mho_{(A,B)}$.}
\end{definition}

\begin{proposition}\label{L-2}
The compatibility relation `$\sim$' on $A\times B$ is an equivalence relation.
\end{proposition}

The relation between the set-indexing number of an edge of a graph $G$ and the compatibility index of the pair of set-labels of its end vertices is established in the following lemma.
 
\begin{lemma}\label{L-3}
Let $f$ be an IASI of a graph $G$ and $u,v$ be two vertices of $G$. Then, $g_f(uv)= f(u)+f(v)=\{a+b:a\in f(u), b\in f(v)\}$. Then, the set-indexing number of the edge $uv$ is $|g_f(uv)| = \mho_{(f(u),f(v))}$.
\end{lemma}
\begin{proof}
By Definition \ref{D-CC}, each $(a,b)\in \mathsf{C}_k$ contributes a single element $k$ to the set $g_f(uv)=f(u)+f(v)$. Therefore, the number of elements in the set $g_f(uv)$ is the number of compatibility classes $\mathsf{C}_k$ in $f(u)\times f(v)$. Hence, $|g_f(uv)| = \mho_{(f(u),f(v))}$.
\end{proof}

Let $r_k$ be the number of elements in a compatibility class $\mathsf{C}_k$. From Lemma \ref{L-3}, we note that only one representative element of each compatibility class $\mathsf{C}_k$ of $f(u)\times f(v)$ contributes an element to the set-label of the edge $uv$ and all other $r_k-1$ elements are neglected and hence we may call these elements {\em neglecting elements} of $\mathsf{C}_k$. The number of such neglecting elements in a compatibility class $\mathsf{C}_k$ may be called the {\em neglecting number} of that class. Then, the neglecting number of $f(u)\times f(v)$, denoted by $r$, is given by $r = \sum{r_{k_1}}$, where the sum varies over distinct compatibility classes in $f(u)\times f(v)$. Hence, we have the following theorem.

\begin{theorem}\label{T-Card}
Let $f$ be an IASI of a graph $G$ and let $u$ and $v$ be two vertices of $G$. Let $|f(u)|=m$ and $|f(v)|=n$. Then, $g_f(uv)=mn-r$, where $r$ is the neglecting number of $f(u)\times f(v)$. 
\end{theorem}

We rewrite Lemma \ref{L-0} as follows and provide an alternate proof for it.

\begin{lemma}\label{L-1}
\cite{GS1} Let $f$ be an IASI on $G$, then for the vertices $u,v$ of $G$, $max(|f(u)|,|f(v)|)\le |g_f(uv)|=|f(u)+f(v)|\le |f(u)|~|f(v)|$.
\end{lemma}
\begin{proof}
Let $|f(u)|=m$ and $|f(v)|=n$. Since $0\le r<mn$, by Observation \ref{T-Card}, $|g_f(uv)|= |f(u)|+|f(v)|\le mn=|f(u)|~|f(v)|$. 

Now, without loss of generality, let $|g_f(uv)|=m$. Then, by Observation \ref{T-Card}, 
\begin{align*}
mn-r=m ~ \Rightarrow  mn=m+r ~ \Rightarrow n|r. \tag{\theequation 1} \label{myeqn1} \\
\text{Also,} ~ mn-m = r ~ \Rightarrow ~ m(n-1)=r ~ \Rightarrow ~ (n-1)|r. \tag{\theequation 2} \label{myeqn2}
\end{align*}

Since $r$ can assume any non-negative integers less than $mn$ and $n$ is a positive integer such that $n|r$ and $(n-1)|r$, we have $r=0$. Hence, by Equation \ref{myeqn2}, $m(n-1)=0 \Rightarrow (n-1)=0$, since $m\ne 0$. Hence, $n=1$. Hence, we have $max(|f(u)|,f(v)|)$. Let $|g_f(uv)|=n$. Therefore, proceeding as above, $m=1$. Combining the above conditions,  $max(|f(u)|,|f(v)|)\le |g_f(uv)|\le |f(u)| |f(v)|$.
\end{proof}

Now, we recall the following definitions.

\begin{definition}\label{D-WIASI}{\rm
\cite{GS1} An IASI $f$ is said to be a {\em weak IASI} if $|g_f(uv)|=max(|f(u)|,|f(v)|)$ for all $u,v\in V(G)$.   A graph which admits a weak IASI may be called a {\em weak IASI graph}. A weak  IASI is said to be {\em weakly uniform IASI} if $|g_f(uv)|=k$, for all $u,v\in V(G)$ and for some positive integer $k$.}
\end{definition}

\begin{definition}\label{D-SIASI}{\rm
\cite{GS2} An IASI $f$ is said to be a {\em strong IASI} if $|g_f(uv)|=|f(u)| |f(v)|$ for all $u,v\in V(G)$. A graph which admits a  strong IASI may be called a {\em strong IASI graph}. A  strong  IASI is said to be  {\em strongly uniform IASI} if $|g_f(uv)|=k$, for all $u,v\in V(G)$ and for some positive integer $k$.}
\end{definition}

Form the above definitions we have the following proposition.

\begin{proposition}
If $f$ is a weak (or strong) IASI defined on a graph $G$, then for each adjacent pair of vertices $u$ and $v$ of $G$, each compatibility class of the pair of set-labels $f(u)$ and $f(v)$ is a trivial class.
\end{proposition}

Do the of the associated graphs, mentioned of above, of a given weak IASI graph admit weak IASI? In the following section, we provide some results regarding this problem.

\subsection{Associated Graphs of Weak IASI graphs}

In \cite{GS3}, it is proved that the graph $G\circ e$ admits a weak IASI if and only if it is bipartite or $e$ is a mono-indexed edge. It is also proved in \cite{GS3} that a graph which is obtained from a cycle $C_n$ by applying finite number of elementary topological reduction on $C_n$ admits a weak IASI. The question we need to address here is whether this result could be extended to any weak IASI graph. The following theorem provides a solution to this problem.

\begin{theorem}
Let $G$ be a weak IASI graph. Let a vertex $v$ of $G$ with $d(v)=2$, which is not in a triangle, and $u$ and $w$ be the adjacent vertices of $v$. Then, the graph $H=(G-v)\cup \{uw\}$ admits a weak IASI if and only if $v$ is not mono-indexed or at least one of the edges $uv$ or $vw$ is mono-indexed. 
\end{theorem} 
\begin{proof}
First assume that $H=(G-v)\cup \{uw\}$ admits a weak IASI, say $f'$, defined as in Theorem \ref{T-IASI-TRG}. Hence, either $u$ or $v$ is mono-indexed in $G$. If both $u$ and $v$ are mono-indexed, then $v$ can have a non-singleton set-label. If exactly one of them is mono-indexed, then $v$ must be mono-indexed. Hence one of the edges $uv$ or $vw$ is mono-indexed.

Conversely, if $v$ is not mono-indexed, the vertices $u$ and $w$ must be mono-indexed. Hence, the edge $uw$ is mono-indexed in $H$. If $v$ is mono-indexed, by hypothesis one the edges $uv$ and $vw$ is mono-indexed. Then, either $u$ or $w$ is mono-indexed. Hence, $H$ admits a weak ISI.  
\end{proof}

We know that the adjacency of two vertices in the line graph $L(G)$ corresponds to the adjacency of corresponding edges in the graph $G$. A vertex in $L(G)$ is mono-indexed if the corresponding edge in $G$ is mono-indexed. Hence, we have the following theorem.

\begin{theorem}\label{T-WIASI-LG}
The line graph $L(G)$ of a weak IASI graph $G$ admits a weak IASI if and only if at least one edge in each pair of adjacent edges in $G$, is mono-indexed in $G$. 
\end{theorem}

The following theorem verifies the admissibility of Weak IASI by the total graph of a weak IASI graph.  

\begin{theorem}\label{T-WIASI-TG}
The total graph $T(G)$ of a weak IASI graph $G$ admits a weak IASI if and only if the IASI defined on $G$ is $1$-uniform. 
\end{theorem}
\begin{proof}
We note that the adjacency between the vertices are preserved in the total graph. That is, the vertices of $T(G)$ corresponding to the vertices of $G$ preserve the same adjacency relations as the corresponding vertices have in $G$. By Theorem \ref{T-WIASI-LG}, the vertices in $T(G)$ corresponding to the edges in $G$ has adjacency between them subject to the condition that at least one edge in each pair of adjacent edges in $G$ is mono-indexed in $G$. The vertices in $T(G))$, that keeps adjacency corresponding to the incidence relation in $G$, hold the condition of a weak IASI that one of the adjacent vertices is mono-indexed if and only if both an edge and the vertices it incidents on are mono-indexed. This completes the proof.
\end{proof}

Figure \ref{G-WIASIG-TG} depicts the admissibility of an IASI by the total graph $T(G)$ of a weakly IASI graph $G$.

\begin{figure}[h!]
\centering
\includegraphics[scale=0.4]{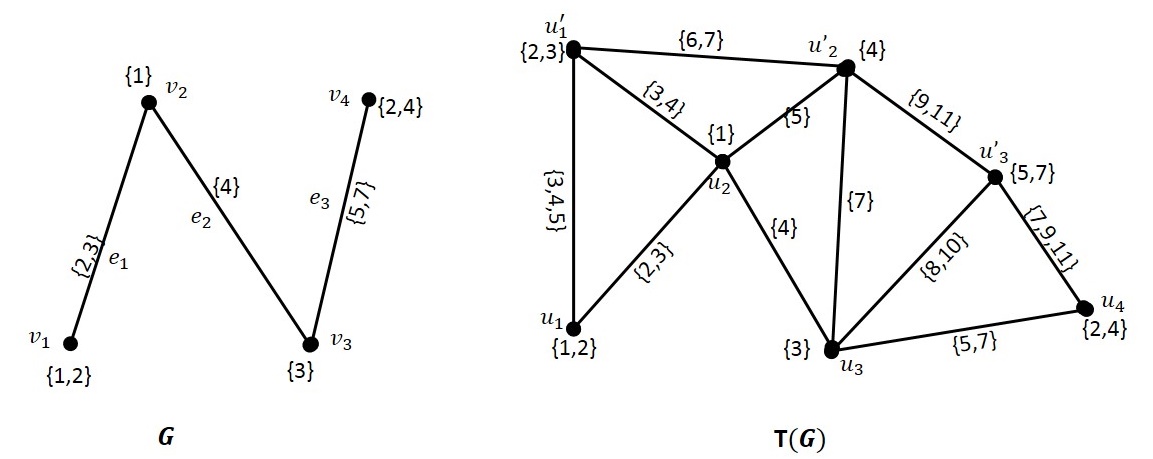}
\caption{}\label{G-WIASIG-TG}
\end{figure}

\begin{remark}{\rm
Form the above theorems, we observe that the line graph and total graph of a weakly uniform IASI graph do not admit weakly uniform IASIs.}
\end{remark}

\subsection{Associated Graphs of Strong IASI graphs}

In \cite{GS2}, a graph which is obtained from a strong IASI graph $G$ by applying finite number of elementary topological reduction on $C_n$ admits a strong IASI if and only if there exists a path $P_2$, not in a cycle, with the difference sets of set-labels of all its vertices are disjoint.

Does the line graph of a strong IASI graph admit a strong IASI? We get the answer for this question from the following theorem.

\begin{theorem}\label{T-SIASI-LG}
The line graph of a strong IASI graph does not admit a strong IASI.
\end{theorem}
\begin{proof}
The line graph of $L(G)$ of a strong IASI graph $G$ admits a strong IASI if for any two adjacent vertices $u_i$ and $u_j$, the difference sets of their set-labels are disjoint. This is possible only when, for the corresponding adjacent edges $e_i$ and $e_j$ in $G$, the maximum of the differences between the elements of the set-label of $e_i$ is less than the minimum of the differences between the elements of the set-label of $e_j$. Since $e_i$ and $e_j$ have a common vertex, this condition will not hold. Hence $L(G)$ does not admit a strong IASI.
\end{proof}

In a similar way,  we establish the following result on total graphs of strong IASI graphs.

\begin{theorem}
The total graph of a strong IASI graph does not admit a strong IASI.
\end{theorem}
\begin{proof}
As provided in the proof of the Theorem \ref{T-SIASI-LG}, the vertices in $T(G)$ corresponding to the edges in $G$ has adjacency between them subject to the condition that the difference sets of their set-labels are disjoint, which is not possible since these edges have a common vertex in $G$ and the differences between some elements in the set-labels of both edges $e_i$ and $e_j$ are the same as that of some elements of the set-label of the common vertex. Hence, $T(G)$ does not admit a strong IASI.
\end{proof}

\section{IASIs with Finite Ground Sets}

The ground set, we have considered so far for the set-labeling of the elements of graphs is $\mathbb{N}_0$, which is a countably infinite set. In the following results, we consider a finite set $X$ of non-negative integers as the ground set. An interesting question here is about the minimum cardinality of the ground set $X$ so that $G$ admits an IASI with respect to $X$. The following proposition leads us to the solution of this problem.

\begin{theorem}\label{P-FinSet}
Let $X$ be a finite set of non-negative integers and let $f:V(G)\to 2^X-\{\emptyset\}$ be an IASI on $G$, which has $n$ vertices. Then, $X$ has at least $\lceil log_2(n+1)\rceil$ elements, where $\lceil x \rceil$ is the ceiling function of $x$.
\end{theorem}
\begin{proof}
Let $G$ be a graph on $n$ vertices. Then, there must be $n$ non-empty subsets of $X$ to label the vertices of $G$. Therefore, $X$ must have at least $n+1$ subsets including $\emptyset$. That is, $2^{|X|}\ge (n+1)$. Hence, $|X|\ge \lceil log_2(n+1)\rceil$.
\end{proof}

\begin{theorem}\label{P-UniSet}
Let $X$ be a finite set of non-negative integers and let $f:V(G)\to 2^X-\{\emptyset\}$ be an IASI on $G$, which has $n$ vertices, such that $V(G)$ is $l$-uniformly set-indexed. Then, the cardinality of $X$ is given by $\binom{|X|}{l} \ge n$.
\end{theorem}
\begin{proof}
Since $V(G)$ is $l$-uniformly set-indexed, $f(u)=l ~\forall~ u\in V(G)$. Therefore, $X$ must have at least $n$ subsets with $l$-elements. Hence, $\binom{|X|}{l} \ge n$.
\end{proof}

\subsection{Open Problems}

We note that an IASI with a finite ground set $X$ is never a topological set-indexer. Is an IASI set-graceful? We note that all IASIs are not set-graceful in general. The following is an open problem for further investigation.

\begin{problem}
Check whether an IASI $f$ can be a set-graceful labeling. If so, find the necessary and sufficient condition for $f$ to be set-graceful.
\end{problem}

\begin{problem}
Check whether an IASI $f$ can be a set-sequential labeling. If so, find the necessary and sufficient condition for $f$ to be set-sequential.
\end{problem}

\section{Conclusion}

In this paper, we have discussed about integer additive set-indexers of graphs and the admissibility of IASI by certain associated graphs of a given IASI graph. Some studies have also been made certain IASI graphs with regard to the cardinality of the labeling sets of their elements. The characterisation of certain types IASIs in which the labeling sets have definite patterns, are yet to be made. Questions related to the admissibility of different types of IASIs by certain graph classes and graph structures are to be addressed. All these facts highlights a wide scope for future studies in this area.

\end{document}